\documentclass[journal]{IEEEtran}
\usepackage{ifpdf}
\ifpdf
    \usepackage{graphicx}
    \usepackage{epstopdf}
\else
    \usepackage{graphicx}

\fi

\usepackage{times} % assumes new font selection scheme installed
\usepackage{amsmath,amsfonts} % assumes amsmath package installed
\usepackage{amssymb}  % assumes amsmath package installed
\usepackage{tabularx}
\usepackage{algorithm}
\usepackage{algorithmic}
\usepackage{ulem}
\usepackage{dsfont}
\usepackage[dvipsnames,usenames]{color}
\usepackage[colorlinks,%
linkcolor=BrickRed,%
filecolor=BrickRed,%
citecolor=RoyalPurple,%
]{hyperref}
\usepackage{caption}
\usepackage{subfigure}

\usepackage{color}

\usepackage{tabularx}
\newcommand{\lr}[2]{\langle #1, #2 \rangle}
\def\Re{\mathbb{R}}

\def\Sec#1{Sec.~\ref{#1}}

\def\notes#1{\marginpar{\tiny #1}\typeout{Notes!
Notes!
Notes!
}}
\renewcommand{\notes}[1]{\typeout{notes!}}

\def\Re{\field{R}}

\def\Sec#1{Sec.~\ref{#1}}
\def\eqdef{\mathrel{:=}}

\def\Sec#1{Sec~\ref{#1}}

\def\Expect{{\sf E}}

\def\Expect{{\sf E}}

\def\Sec#1{Sec.~\ref{#1}}

\def\IEEEQEDclosed{\mbox{\rule[0pt]{1.3ex}{1.3ex}}}

\def\eqdef{\mathrel{:=}}

\newtheorem{theorem}{Theorem}
\newtheorem{example}{Example}

\newtheorem{lemma}{Lemma}
\newtheorem{remark}{Remark}
\newtheorem{proposition}{Proposition}
\newtheorem{corollary}{Corollary}
\newtheorem{assumption}{Assumption}

\def\beq{\begin{eqnarray}} 
\def\bc{\begin{center}} 
\def\be{\begin{enumerate}}
\def\bi{\begin{itemize}} 
\def\bs{\begin{small}}
\def\bS{\begin{slide}}
\def\ec{\end{center}} 
\def\ee{\end{enumerate}}
\def\ei{\end{itemize}}
\def\es{\end{small}}
\def\eS{\end{slide}}
\def\eeq{\end{eqnarray}}

\newcommand{\ud}{\,\mathrm{d}}

\def\Re{\mathbb{R}}

\def\Sec#1{Sec.~\ref{#1}}

\def\Expect{{\sf E}}

\renewcommand{\Re}{\mathbb{R}}

\def\eqdef{\mathbin{:=}}

\newcommand{\Teps}{{T}_{\epsilon}}

\newcommand{\geps}{{g}_{\epsilon}}

\newcounter{rmnum}
\newenvironment{romannum}{\begin{list}{{\upshape (\roman{rmnum})}}{\usecounter{rmnum}
\setlength{\leftmargin}{14pt}
\setlength{\rightmargin}{8pt}
\setlength{\itemsep}{2pt}
\setlength{\itemindent}{-1pt}
}}{\end{list}}

\newcounter{anum}

\newcommand{\calX}{\mathcal{X}}
\newcommand{\calB}{\mathcal{B}}
\title{\LARGE \bf
	On the Lyapunov Foster criterion and Poincar\'e inequality \\ for
      Reversible Markov Chains}

\author{Amirhossein Taghvaei and Prashant G. Mehta
\thanks{A. Taghvaei is with the Department of Mechanical and Aerospace Engineering at University of California, Irvine
\tt\small ataghvae@uci.edu} 
\thanks{P. G. Mehta is with the Department of Mechanical Science and Engineering at the University of
Illinois at Urbana-Champaign (UIUC)
\tt\small  mehtapg@illinois.edu}% <-this % stops a space,
\thanks{Financial support from the NSF grant 1761622
	and the ARO grant W911NF1810334 is gratefully acknowledged.}%
}

\begin{document}
\normalem
\maketitle

\begin{abstract}
This paper presents an elementary proof of stochastic stability of a
discrete-time reversible Markov chain starting from a Foster-Lyapunov
drift condition.  Besides its relative simplicity, there are two
salient features of the proof: (i) it relies entirely on
functional-analytic non-probabilistic arguments; and (ii) it makes
explicit the connection between a Foster-Lyapunov function and
Poincar\'e inequality.  The proof is used to derive an explicit bound
for the spectral gap.  An extension to the non-reversible case is also
presented.  

\end{abstract}

\section{Introduction}

This paper presents an elementary functional-analytic proof of
stochastic stability of a discrete-time reversible Markov chain. The
main hypothesis is the existence of a Foster Lyapunov function, drift
condition (v4) in~\cite[Ch. 15, Ch. 16]{MT}. 
The main result is to establish Poincar\'e inequality and relate it to
a spectral gap under additional hypothesis.  The spectral gap yields
geometric convergence as an easy consequence.

The use of Lyapunov drift condition (v4) to establish geometric
convergence rate is standard in the theory of Markov chains;
cf.,~\cite{MT} and references therein.  It is known that the geometric
ergodicity is equivalent to a spectral gap for the corresponding
Markov operator in a certain normed vector space
$L^V_\infty$~\cite{kontoyiannis2003spectral}.  
The spectral gap in $L^V_\infty$ implies a spectral gap in $L^2$ for
reversible Markov chains~\cite{roberts1997geometric}.  Explicit bounds
on the convergence rate are obtained
in~\cite{meyn1994computable,rosenthal1995minorization}. However, in a
general setting, the existing bounds can be difficult to compute.

% The Lyapunov drift condition is also extended to continuous-time Markov chains~\cite{down1995exponential,kontoyiannis2005large}. 

The techniques and tools used in~\cite{MT} and the related literature
are probabilistic in nature.  In contrast, the short proof in this
paper is entirely analytical and relies on elementary arguments.  The
key is to use the Lyapunov Foster condition (v4) to derive a Poincar\'e
inequality.  This is then related to existence of the spectral gap from
which the convergence result follows.  
The approach of this paper is inspired
by~\cite{bakry08,bakry2008rate}~\cite[Ch. 4]{bakry2013} where Lyapunov
function is related to Poincar\'e
inequality for a continuous-time Markov processes.
To the best of our knowledge, the extension of this connection,
between Poincar\'e inequality and Lyapunov function, in discrete-time
setting is not known.  Given the
elementary nature of the proof and the explicit bound on spectral gap,
the results of this paper are expected to be broadly useful to the practitioners who use the discrete-time reversible Markov chain
for Markov chain Monte-Carlo (MCMC) and  simulation purposes.    

% In particular, the extension is
% used to derive explicit bounds on the spectral gap which is an
% important problem in its own right with implications to MCMC and other
% simulation techniques.  

Analysis of geometric ergodicity based on Lyapunov drift condition
appears in~\cite{hairer2011yet}. Their main
result~\cite[Thm. 1.3]{hairer2011yet} is based on introducing a family
weighted normed spaces $L^{\beta V}_\infty$ and establishing the
spectral gap in this space, for a particular weight $\beta$. This is
different compared to this paper where a direct connection between
Lyapunov condition and Poincar\'e  inequality is established, and
explicit bounds on the $L^2$ spectral gap are derived. 

%However, the result in~\cite{hairer2011yet} is stronger as it is applicable to non-reversible Markov chains as well.

The outline of the remainder of this paper is as follows: The preliminaries and problem
statement appears in \Sec{sec:prelim}.  The main result for the
reversible Markov chain appears in \Sec{sec:main}.  Some
extensions to reversible and to non-reversible cases
are discussed in \Sec{sec:extension}. The main result is illustrated
with examples in \Sec{sec:example}. Some concluding remarks appear
in \Sec{sec:conclusion}.
\section{Preliminaries}\label{sec:prelim}
\subsection{Model and definitions}
\label{sec:setup}
Consider a time-homogeneous discrete-time Markov process
$\{X_n\}_{n\geq 0}$ taking values in Polish state space
$\calX$, equipped with the Borel $\sigma$-field $\mathcal{B}$.  Let $P$ denote the corresponding Markov operator defined such that \[P f(x) = \Expect[f(X_1)|X_0=x],\] for all bounded measurable functions $f:\calX \to \Re$. 
Let $p:\calX\times \mathcal{B} \rightarrow [0,1]$ be the probability
transition kernel associated with $P$. In terms of this kernel, the
action of $P$ on bounded measurable functions as follows:
\begin{equation*}
P f(\cdot ) = \int_{\calX} f(y)p(\cdot ,\ud y).
\end{equation*}
The action of $P$ on probability measure $\mu$ on $(\calX,\calB)$ is
as follows:
\begin{equation*}
\mu P (\cdot) = \int_{\calX} p(x ,\cdot)\ud \mu(x).
\end{equation*}
%The Markov operator $P$ is represented with a probability kernel $p(\cdot,\cdot)$ according to 
%where $p(\cdot,\cdot )$ is the probability kernel (that is for all $x\in \Re^d$, $p(x,\cdot)$ is a probability measure, and for all measurable sets $A \in \mathcal{B}(\Re^d)$, $p(\cdot,A)$ is measurable). 
%One can also 
A probability measure $\pi$  is  invariant for $P$ if $\pi P = \pi$.
%\[
%\pi 
%\]
%for all bounded measurable functions $f$.

Consider the space of square integrable functions with respect to $\pi$ denoted as $L^2(\pi)$ equipped with the inner product
\begin{equation*}
\lr{f}{g}_{\pi} := \int f(x)g(x) \ud \pi(x),
\end{equation*} 
%\begin{definition}
and the norm $\|f\|^2_{2,\pi} := \lr{f}{f}_{\pi}$. It follows from Jensen's inequality that $P$ is a bounded linear operator on $L^2(\pi)$, when $\pi$ is the invariant measure~\cite[pp. 10]{bakry2013}.  The
invariant measure $\pi$ is said to be reversible for the Markov
operator $P$ if 
$P$ is self-adjoint on $L^2(\pi)$, i.e., 
\begin{equation*}
\lr{f}{Pg}_{\pi}= \lr{Pf}{g}_{\pi},\quad \forall f,g \in L^2(\pi). 
\end{equation*}

In this paper, we consider Markov chains $P$ with a unique reversible
invariant measure $\pi$, formalized below as an assumption:

\medskip

\begin{assumption}\label{asmpt:reversible} $P$ admits a unique reversible invariant measure $\pi$.  
\end{assumption}

\medskip

The main question is to establish a spectral gap (in $L^2(\pi)$) for
$P$. Since $P$ has an eigenvalue $\lambda=1$ with
eigenfunction $f(x)\equiv1$, we consider the orthogonal subspace
$L^2_0(\pi)=\{f\in L^2(\pi);~\int f(x)\ud \pi(x)=0\}$. 
$P$ is said to admit a spectral gap $\beta>0$ in $L^2_0(\pi)$ if
\begin{equation}\label{eq:spectral-gap-def}
\|P\|_{L^2_0(\pi)}=\sup_{f \in L^2_0(\pi)}~ \frac{\|Pf\|_{2,\pi}}{\|f\|_{2,\pi}}\leq 1- \beta.
\end{equation}    

Two immediate consequences of the spectral gap are as follows:
\begin{enumerate}
\item Geometric convergence of the moments in $L^2(\pi)$
\begin{equation*}
\|P^n f - \pi(f)\|_{2,\pi}\leq (1-\beta)^n \|f - \pi(f)\|_{2,\pi},
\end{equation*}
 where $\pi(f) := \int f(x) \ud \pi(x)$ is the mean of $f$ with respect to the invariant measure $\pi$. 
\item Geometric convergence of the probability distribution in the
  total-variation distance~\cite[Thm. 2.1]{cattiaux2009trends},
\begin{equation*}
\|\mu P^n - \pi\|_{TV} \leq (1-\beta)^n \|h -1\|_{2,\pi},
\end{equation*}
for any initial distribution $\ud \mu = h \ud \pi$.   
\end{enumerate}

For reversible Markov chains, the spectral gap is related to the
Poincar\'e inequality as explained in the following section.

\subsection{Spectral gap and Poincar\'e inequality}
%The spectral gap is closely related to the Poincar\'e inequality. 
Define the Dirichlet forms 
\begin{align*}
\mathcal{E}(f,f) \eqdef \lr{f}{(I-P)f}_\pi,\quad \tilde{\mathcal{E}}(f,f) \eqdef \lr{f}{(I+P)f}_\pi. %= \frac{1}{2} \int_{\calX \times \calX} p(x,\ud y )(f(x)-f(y))^2 \ud \pi(x)
\end{align*}
Then $P$ is said to satisfy the Poincar\'e inequality, if there are positive constants $\beta_+$ and $\beta_-$ such that 
%Define the following two inequalities
\begin{align}
\|f\|^2_{2,\pi}\leq \frac{1}{\beta_+}\mathcal{E}(f,f)  ,\quad \forall f \in L^2_0(\pi), \label{eq:poincare_+}\\
\|f\|^2_{2,\pi}\leq \frac{1}{\beta_-}\tilde{\mathcal{E}}(f,f)  ,\quad \forall f \in L^2_0(\pi).\label{eq:poincare_-}
\end{align}

\begin{lemma}
	Under Assumption~\ref{asmpt:reversible},
	\begin{romannum}
		\item If $P$ satisfies the Poincar\'e inequality~\eqref{eq:poincare_+} with constant $\beta_+>0$, then the spectrum of $P$ on $L^2_0(\pi)$ is bounded above by $1-\beta_+$. 
		\item If $P$ satisfies the Poincar\'e inequality~\eqref{eq:poincare_-} with constant $\beta_->0$, then the spectrum of $P$ on $L^2_0(\pi)$ is bounded below by $-1+\beta_-$. 
		\item If $P$ satisfies the Poincar\'e inequalities~\eqref{eq:poincare_+}-\eqref{eq:poincare_-}, then  it admits spectral gap $\beta = \min(\beta_+,\beta_-)$.  
	\end{romannum} 
\end{lemma}

\proof{Omitted.  See~\cite[Sec. 5.2.1, pp. 115]{stroock2013introduction}}
%Then, one can conclude  that $P$ admits the spectral gap $C=\min(C_-,C_+)$. 
%% where 
%The relationship  follows 
%because $P$ is self-adjoint, and the norm of $P$ is equivalently given by
%\begin{equation*}
%\|P\|_{L^2_0(\pi)}=\sup_{f \in L^2_0(\pi)}~ \frac{|\lr{f}{Pf}|}{\lr{f}{f}}%\leq 1- C
%\end{equation*} 
%Then, define the following two inequalities
%\begin{align}\label{eq:poincare}
%\|f-\hat{f}\|^2_{L^2(\pi)}\leq \frac{1}{C}\lr{f}{(I-P)f} ,\quad \forall f \in L^2(\pi) 
%\end{align}

%
%\begin{equation}\label{eq:poincare}
%\|f-\hat{f}\|^2_{L^2(\pi)}\leq \frac{1}{C}\lr{f}{(I-P)f} ,\quad \forall f \in L^2(\pi) 
%\end{equation}
%for a positive constant $C>0$, where $\hat{f} = \int f \ud \pi$. 

\medskip
\begin{remark}\label{rem:spec}
	% In continuous-time setting, the Poincar\'e
        % inequality~\eqref{eq:poincare_+} is sufficient to establish
        % spectral gap.  
	For a continuous-time reversible Markov process with the infinitesimal
        generator $L$ and the semigroup $P_t=e^{tL}$, the Dirichlet form
        is defined as $\mathcal E(f,f) := -\lr{f}{Lf}_\pi$.
        Therefore, the
        Poincar\'e inequality~\eqref{eq:poincare_+} is expressed as
        $\lr{f}{Lf}_\pi\leq  -\beta_+ \|f\|_{2,\pi}^2$, from which the
        spectral gap for the semigroup,
        $\|P_t\|_{L^2(\pi)}=\|e^{tL}\|_{L^2(\pi)} \leq e^{-\beta_+ t}
        <1$, readily follows.  In discrete-time settings, the second
        Poincar\'e inequality~\eqref{eq:poincare_-} is also required.
        This is to rule out periodicity, eigenvalue at $-1$ for the
        reversible case~\cite[Ch. 5]{stroock2013introduction}.   
\end{remark}

\section{Main Result}\label{sec:main}

% \subsection{Poincar\'e inequality and Lyapunov conditions}
% Our analysis is based on the following assumption:
% \medskip

The main hypothesis of the paper is the Foster Lyapunov condition
(v4):
\begin{assumption}\label{asmpt:Lyapunov}
	Suppose $P$ satisfies
	\begin{align}
	PV &\leq (1-\lambda)V + b\mathds{1}_K,\label{eq:Lyapunov-condition}\\
	%p(x,A)
	P\mathds{1}_A(x) &\geq \alpha \nu(A) \mathds{1}_{K}(x),\quad \forall A \in \mathcal{B},\label{eq:assumption2}
	\end{align}
	for a positive function $V:\Re^d \to [1,\infty)$, numbers $b<\infty$, $\alpha,\lambda>0$, a set $K \subseteq \calX$, and a probability measure $\nu$. 
\end{assumption} 
\medskip
The condition~\eqref{eq:Lyapunov-condition} is known as the drift
condition and condition~\eqref{eq:assumption2} is known as the
minorization condition.  
The main result of this paper is as follows:
%Next, we show that the Poincar\'e inequality~\eqref{eq:poincare_+} follows from Assumption~\ref{asmpt:Lyapunov}. 
%\end{definition}
%\medskip
%\noindent
%{\bf Assumption:} The Markov operator $P$ admits an invariant measure $\pi$ which is also reversible 
\medskip
%  $f \in C_b(\Re^d)$.  
%\section{Lyapunov condition implies the spectral gap}
%Now, we are ready to state the mai
\begin{theorem}\label{thm:Lyapunov-Poincare}
%	Let $P$ be a Markov operator that satisfies 
Under Assumptions~\ref{asmpt:reversible}-\ref{asmpt:Lyapunov},
%	\begin{align}
%	PV &\leq (1-\lambda)V + b\mathds{1}_K\label{eq:Lyapunov-condition}\\
%	p(x,A) &\geq \alpha \nu(A) \mathds{1}_{x\in K},\quad \forall A \in \mathcal{B}(\Re^d)\label{eq:assumption2}
%	\end{align}
%	for a positive function $V:\Re^d \to [1,\infty)$, numbers
%	$b<\infty$, $\alpha,\lambda>0$, compact set $K \subset \Re^d$,
%	and a probability measure $\nu$. 
$P$ admits a Poincar\'e inequality~\eqref{eq:poincare_+} with constant
%	\begin{equation}\label{eq:spectral-gap}
%	\|P\|_{L^2_0(\pi)} \leq 1-C
%	\end{equation}
%	where 
	$\beta_+=\frac{\lambda}{1+\frac{2b}{\alpha}}$. % If $P$ is
        % further assumed to be a positive semi-definite operator (on
        % $L^2(\pi)$ then $P$ also admits spectral gap $\beta = \beta_+
        % =\frac{\lambda}{1+\frac{2b}{\alpha}}$. 
\end{theorem}
\medskip

For continuous-time processes, the analogous result appears
in~\cite[Thm. 4.6.2 pp. 202]{bakry2013}.  
% Theorem~\ref{thm:Lyapunov-Poincare} extends the result on establishing Poincar\'e inequality with Lyapunov function method for continuous-time processes  to discrete-time setting. However, unlike 
Unlike the continuous-time case, the conclusion of the
Theorem~\ref{thm:Lyapunov-Poincare} is not sufficient to establish a
spectral gap without also establishing~\eqref{eq:poincare_-}, except
for the case when $P$ is positive semi-definite.  
%One can have~\eqref{eq:poincare_-} for free when the Markov operator
%is positive semi-definite.  

\medskip
\begin{corollary}\label{cor:positive}
	Under the assumptions of Theorem~\ref{thm:Lyapunov-Poincare},
        if $P$ is a positive semi-definite operator, then $P$ admits a
        spectral gap $\beta = \beta_+ =\frac{\lambda}{1+\frac{2b}{\alpha}}$. 
\end{corollary}

\medskip
Later, in \Sec{sec:P2}, additional assumption is introduced to
establish spectral gap for Markov operators that are not necessarily
positive semi-definite.

%\begin{remark}
%Theorem~\ref{thm:Lyapunov-Poincare} extends the result about establishing Poincar\'e inequality with Lyapunov function method~\cite[Thm. 4.6.2 pp. 202]{bakry2013}  to discrete-time setting. 
%\end{remark}
%\subsection{Poincar\'e inequality and spectral gap}
%Now that we established the Poincar\'e  inequality, it remains to show that the Poincar\'e inequality implies the spectral gap for $P$. However, this is not possible as we discuss below. 

%Because $P$ is self-adjoint, the norm of $P$ is given by
%\begin{equation*}
%\|P\|_{L^2_0(\pi)}=\sup_{f \in L^2_0(\pi)}~ \frac{|\lr{f}{Pf}|}{\lr{f}{f}}%\leq 1- C
%\end{equation*} 
%The Poincar\'e inequality with constant $C$ implies that 
%\begin{equation*}
%\sup_{f \in L^2_0(\pi)}~ \frac{\lr{f}{Pf}}{\lr{f}{f}}\leq 1- C
%\end{equation*}    
% i.e the spectrum of $P$ is bounded from above by $1-C$. In order to show the spectral gap, we also need to show that the spectrum is bounded from below by $-1+C$.  

%\medskip
 
\subsection{Proof of Theorem~\ref{thm:Lyapunov-Poincare}}

\begin{remark}
If $K= {\cal X}$ then the minorization condition~\eqref{eq:assumption2}
is the Doeblin's condition which directly implies the spectral gap
$\|P\|_{L^2_0(\pi)} \leq
1-\frac{\alpha}{2}$~\cite[Sec. 2.2, pp. 28]{stroock2013introduction}.  
However, Doeblin's condition is a very strong assumption for 
Markov-chains. 
% and condition~\eqref{eq:assumption2} holds on a compact set. 
%	
%However, in general, the condition~\eqref{eq:assumption2} is expected to hold in a compact set. 
In the other extreme when $K = \emptyset$ then the drift
condition~\eqref{eq:Lyapunov-condition} implies the spectral gap
$\|P\|_{L^2_0(\pi)} \leq 1-\lambda$ from the spectral theory of
positive operators~\cite[Thm. 13.1.6, pp. 383]{davies2007linear}.
Owing to the eigenvalue at $1$, this case does
not apply to Markov operators.  However, by suitably adapting the
proof from the theory of positive operators to accomodate the
minorization condition~\eqref{eq:assumption2}, one obtains an
elementary proof of Theorem~\ref{thm:Lyapunov-Poincare} as presented
next.  
\end{remark}

%\end{remark}
\medskip
\begin{proof}[Proof of the Theorem~\ref{thm:Lyapunov-Poincare}]
%	The spectral gap~\eqref{eq:spectral-gap} is equivalent to the Poincar\'e inequality~\cite[Sec. 5.2.1]{stroock2013introduction}
%	\begin{equation*}
%	\|f-\hat{f}\|^2_{L^2(\pi)}\leq \frac{1}{C}\lr{f}{(I-P)f} ,\quad \forall f \in L^2(\pi) 
%	\end{equation*}
%	where $\hat{f} = \int f \ud \pi$. 
	From the variational characterization of the mean, we have
        $\|f-\pi(f)\|_{2,\pi}\leq \| f - m \|_{2,\pi}$ for all
        constants $m\in\Re$. Therefore, in order to prove the Poincar\'e inequality~\eqref{eq:poincare_+},  it suffices to show that  
	\begin{equation}\label{eq:temp-PE}
	\|f-m\|^2_{2,\pi}\leq \frac{1}{C}\lr{f}{(I-P)f}_\pi ,\quad \forall f \in L^2(\pi), 
	\end{equation}
	for some constant $m=m(f)$ to be chosen later. 

Consider the Lyapunov drift condition~\eqref{eq:Lyapunov-condition}. Multiply both sides by $\frac{(f-m)^2}{V}$ to obtain
	\begin{align*}
	\frac{(f-m)^2}{V}PV &\leq (1-\lambda) (f-m)^2 + b\frac{1}{V}(f-m)^2 \mathds{1}_K \\&\leq (1-\lambda) (f-m)^2 + b(f-m)^2 \mathds{1}_K,
	\end{align*}  
	where the second inequality follows because $V\geq 1$. Rearranging the terms
	\begin{equation*}
	\lambda (f-m)^2  \leq \frac{(f-m)^2}{V}(I-P)V + b(f-m)^2 \mathds{1}_K, 
	\end{equation*}
	and integrating both sides with respect to $\pi$,
	\begin{equation*}
	\lambda \|f-m\|^2_{2,\pi} \leq \lr{\frac{(f-m)^2}{V}}{(I-P)V}_\pi  + b\|(f-m)\mathds{1}_K\|^2_{2,\pi}. 
	\end{equation*}

It is claimed that
	\begin{align}
	\lr{\frac{(f-m)^2}{V}}{(I-P)V}_\pi &\leq \lr{f}{(I-P)f}_\pi,\label{eq:claim1}\\
	\|(f-m)\mathds{1}_K\|_{2,\pi}^2 &\leq \frac{2}{\alpha} \lr{f}{(I-P)f}_\pi,\label{eq:claim2} 
	\end{align}
	with $m=\frac{1}{\pi(K)}\int_{K} f \ud \pi$. If the claims are true then
	\begin{equation*}
	\|f-m\|^2_{2,\pi}\leq \frac{1+\frac{2b}{\alpha}}{\lambda} \lr{f}{(I-P)f}_\pi,
	\end{equation*}
	which proves~\eqref{eq:temp-PE}, hence the Poincar\'e
        inequality~\eqref{eq:poincare_+} with $\beta_+ =\frac{\lambda}{1+\frac{2b}{\alpha}}$.  It remains to prove
        the two claims:
%	 and the spectral gap $C=\frac{\lambda}{1+\frac{2b}{\alpha}}$. It remains to prove the claims. 
\begin{enumerate}
\item Proof of the claim~\eqref{eq:claim1}: Let $g=f-m$. Then, using $(I-P)1=0$,~\eqref{eq:claim1} is equivalently expressed as 
	\begin{equation}
	\lr{g}{Pg}_\pi \leq \lr{\frac{g^2}{V}}{P V}_\pi. \label{eq:claim11}
	\end{equation}
%	The equivalence follows because the RHS of~\eqref{eq:claim1} is
%	\begin{align*}
%	\lr{f}{(I-P)f} &= 	\lr{g+m}{(I-P)(g+m)} \\
%	&=\lr{g+m}{(I-P)g}\\
%	&=  \lr{g}{(I-P)g} + m\lr{1}{(I-P)g}\\
%	&=  \lr{g}{(I-P)g} + m\lr{(I-P)1}{g}\\
%	&=\lr{g}{(I-P)g} %+ m\lr{g}{(I-P)1}
%	\end{align*} 
%	where we used $(I-P)1 = 0$ and symmetry of $P$. And, the LHS of~\eqref{eq:claim1} is
%	\begin{align*}
%	\lr{\frac{(f-m)^2}{V}}{(I-P)V}  &= \lr{\frac{g^2}{V}}{(I-P)V} \\&=  \lr{\frac{g^2}{V}}{V} -  \lr{\frac{g^2}{V}}{PV}  \\
%	&=\lr{g}{g} -  \lr{\frac{g^2}{V}}{PV} 
%	\end{align*}
	Note that
	\begin{align*}
	0&\leq \int\int  V(x)V(y)\left(\frac{g(y)}{V(y)}-\frac{g(x)}{V(x)}\right)^2 p(x,\ud y)  \ud \pi(x)\\&= \int \int  \frac{g(x)^2}{V(x)}V(y)p(x,\ud y)  \ud \pi(x) \\&\quad\quad+ \int \int \ \frac{g(y)^2}{V(y)}V(x)p(x,\ud y)  \ud \pi(x)\\
	&\quad \quad -2\int \int g(x)g(y)p(x,\ud y)  \ud \pi(x)\\
	&= \lr{\frac{g^2}{V}}{PV}_\pi + \lr{V}{P\frac{g^2}{V}}_\pi -2\lr{g}{Pg}_\pi. 
	\end{align*}
	Using the self-adjoint property of $P$, it follows that 
$
\lr{\frac{g^2}{V}}{PV} = \lr{V}{P\frac{g^2}{V}}
$ which in turn proves~\eqref{eq:claim11}. 
% and as a result 
  %       \begin{equation*}
  %       0 \leq 2 \lr{\frac{g^2}{V}}{PV}  -2\lr{g}{Pg} 
  %       \end{equation*} 
  %       proving the claim.
	
\medskip

\item Proof of the claim~\eqref{eq:claim2}: Note that 
	\begin{align*}
	\lr{f}{(I-P)f}_\pi&= \frac{1}{2} \int\! \int (f(x)-f(y))^2p(x,\ud y)  \ud \pi(x)\\
	&\overset{(1)}{\geq} \frac{\alpha }{2} \int_{K} \!\int  (f(x)-f(y))^2 \ud \nu(y)\ud \pi(x) \\
	&\overset{(2)}{\geq} \frac{\alpha}{2} \int_K \!(f(x)-\int f(y) \ud \nu(y) )^2 \ud \pi(x) \\
	&\overset{(3)}{\geq} \frac{\alpha}{2}\int_K (f(x)-m)^2 \ud \pi(x). 
	\end{align*}
	The first inequality follows from the use of the 
        minorization condition~\eqref{eq:assumption2}. The second
        inequality is 
        % follows because $\int (f(y)-c)^2 \ud \nu(y) \geq (\int f(y)\ud
        % \nu(y) -c)^2$
        the Jensen's inequality. The third
        inequality 
        follows from the variational characterization of the variance
        of the function $f$ 
        because $m= \frac{1}{\pi(K)}\int_K f \ud \pi $ is the
        mean. 
%Hence the claim is proved.  
\end{enumerate}
	%therefore the spectral gap holds with $\lambda=\delta$.
	%\begin{equation*}
	%\int p(x,y)(f(x)-f(y))
	%\end{equation*}
\end{proof}

\subsection{A counter-example}

The following counter-example serves to show that it is not
possible to obtain a bound for $\beta_-$ using {\em only} the Foster
Lyapunov condition (v4).
% In the following example, we show that the spectral gap of $P$ can be
% made arbitrary small, while the constants in the
% Assumption~\ref{asmpt:Lyapunov} do not change. 

 \medskip
 \begin{example}\label{counter-example}
 	Consider 
% a Markov-chain taking two values in the finite state-space $\mathbb{S} = \{1,2\}$. 
 	% Suppose
        the Markov transition matrix \[P = \begin{bmatrix}
 	\epsilon & 1-\epsilon \\
 	1- \epsilon & \epsilon
 	\end{bmatrix},\] on the state-space $\mathbb{S} = \{1,2\}$.
      The invariant measure $\pi = [
      \frac{1}{2},~\frac{1}{2}]$ is reversible. The eigenvalues of $P$ are
      $\lambda = 1, -1 + 2\epsilon$. Therefore, $\beta_+=1-2\epsilon$
      and $\beta_-=2\epsilon$.  In the following, we show that the
      conditions~\eqref{eq:Lyapunov-condition}-\eqref{eq:assumption2}
      hold, with constants that are independent of $\epsilon$.  As a
      result, it is not possible to derive a bound on $\beta_-$ simply
      from the constants that appear in the
      conditions~\eqref{eq:Lyapunov-condition}-\eqref{eq:assumption2}.   
% 	 yielding the Poincar\'e constant $\beta_+ = \frac{1}{14}$, for all $\epsilon\leq \frac{1}{4}$:
 \begin{enumerate}	
\item 	Let the subset $K = \{1\} \subset \mathbb{S}$. 
 	Then the condition~\eqref{eq:assumption2} holds with $\alpha =
        1$ and $\nu = [\nu_1,~\nu_2]=[\epsilon,~1-\epsilon]$ because 
 	\begin{align*}
 	P_{11}  = \epsilon = \alpha \nu_1 ,\quad  P_{12} =1-\epsilon= \alpha \nu_2. 
 	\end{align*}
 \item For all $\epsilon \leq
        \frac{1}{4}$, the condition~\eqref{eq:Lyapunov-condition} holds with $V = [1,~3]^\top$, $\lambda = \frac{1}{2}$, and $b=3$ because 
 	\begin{align*}
 	&(i=1)~~ P_{11}V_1 + P_{12}V_2 = \epsilon + 3(1-\epsilon) \\&\quad \quad \quad \quad \quad \leq \frac{1}{2} + 3 = (1-\lambda)V_1+b,\\
 	&(i=2)~~ P_{21}V_1 + P_{22}V_2 = 1-\epsilon + 3\epsilon \leq \frac{3}{2}=  (1-\lambda)V_2,
 	\end{align*}
%	where the second line holds when 
The resulting bound for $\beta_+$ is
        $\frac{\frac{1}{2}}{1+\frac{2*3}{1}}= \frac{1}{14}$.
        % [AMIR:
        %Fix top line].
\end{enumerate}
 \end{example}

\section{Extensions}\label{sec:extension}

By imposing additional conditions, the analysis of this paper
is useful to obtain bounds for the spectral gap in the
reversible and also the non-reversible cases.  Two sets of results are
described next.  
% The proofs appear in the final subsection after the
% results have been presented.  

\subsection{Spectral gap under stronger condition}\label{sec:P2}
From Corollary~\ref{cor:positive}, a spectral gap is obtained whenever
$P$ is positive semi-definite.  Therefore, one way to prove the
spectral gap for $P$ is to consider $P^2$ which is always positive
semi-definite for reversible Markov processes.  Given the
counter-example~\ref{counter-example}, it is not true that $P^2$
satisfies condition (v4) (if $P$ does), without imposing some additional
condition.  One such condition, based on~\cite[Assumption 2]{hairer2011yet}, is
as follows:  

% To apply the result of Theorem~\ref{}, one needs to show that $P^2$
% satisfies the
% conditions~\eqref{eq:Lyapunov-condition}-\eqref{eq:assumption2} with
% a constant $\tilde{\beta}$.  In that case, $\|P^2\|_{L^2_0(\pi)}\leq
% 1-\tilde{\beta}$, which in turn implies $\|P\|_{L^2_0(\pi)}\leq
% (1-\tilde{\beta})^{1/2}$. 

% To prove such a result for $P^2$, we consider a stronger version of the Assumption~\ref{asmpt:Lyapunov} borrowed from~\cite{hairer2011yet}.  

\medskip
\begin{assumption}\label{asmpt:level-set}
	 In condition (v4) (in Assumption~\ref{asmpt:Lyapunov}), the set 
\[K=\{x \in
         \calX;~V(x)\leq  R\},
\] 
for some $R>\frac{2b}{\lambda}$. %In words, $K$ is a level-set of $V$.  
\end{assumption}
\medskip
%We add an additional condition that

%\begin{align}
%PV &\leq (1-\lambda)V + b\label{eq:Lyapunov-condition-2}\\
%p(x,A) &\geq \alpha \nu(A) \mathds{1}_{x\in K},\quad \forall A \in \mathcal{B}(\Re^d)\label{eq:minorization-2}
%\end{align}
%for a positive function $V:\Re^d \to [1,\infty)$, numbers $b<\infty$, $\alpha,\lambda \in (0,1)$,  a probability measure $\nu$, and $K=\{x \in \Re^d;~V(x)\leq  R\}$ for some $R>\frac{2b}{\lambda}$. 
\begin{proposition}\label{lem:P-squared}
Suppose $P$ is a Markov operator that satisfies
Assumptions~\ref{asmpt:reversible}, \ref{asmpt:Lyapunov},
and~\ref{asmpt:level-set}. Then $P^2$ satisfies 
\begin{align}
P^2V &\leq (1-\lambda')V + b' \mathds{1}_K, \label{eq:Lyapunov-condition-2}\\
P^2 \mathds{1}_A(x) &\geq \alpha' \nu(A) \mathds{1}_{K}(x),\quad \forall A \in \mathcal{B}\label{eq:minorization-2}
\end{align}
where $\lambda'=\lambda(\frac{3}{2}-\lambda)$, $b'=(2-\lambda) b$, and
$\alpha' = \alpha^2\nu(K)$.  Consequently,
\begin{equation}\label{eq:spec-gap-result}
\|P\|_{L^2_0(\pi)}\leq \left(1-\beta^+\right)^{\frac{1}{2}}.
\end{equation}
with
$\beta^+=\frac{\lambda(\frac{3}{2}-\lambda)}{1+\frac{2b(2-\lambda)}{\alpha^2\nu(K)}}$.  
\end{proposition}

\begin{proof}
	Proof of~\eqref{eq:Lyapunov-condition-2}: Because $P$ is a
        positive operator\footnote{An operator $P$ is positive if
          $Pf\geq 0$ whenever $f\geq 0$.}, $Pf\leq Pg$ whenever $f\leq
        g$. Therefore, applying $P$ to both sides of the inequality~\eqref{eq:Lyapunov-condition}, 
	\begin{align*}
	P^2V &\leq (1-\lambda) PV + b P\mathds{1}_K \\
	&\leq (1-\lambda)^2 V + (1-\lambda) b \mathds{1}_K + b P\mathds{1}_K \\
	&\overset{(3)}{\leq}  (1-\lambda)^2 V + (1-\lambda) b \mathds{1}_K+ b(\mathds{1}_K + \frac{V}{R}) \\
	&\overset{(4)}{\leq}  ((1-\lambda)^2+\frac{\lambda}{2}) V + (2-\lambda) b \mathds{1}_K,
	\end{align*}
	where the third inequality follows from $P\mathds{1}_K \leq P
        \mathds{1} =\mathds{1}\leq  \mathds{1}_K + \frac{V}{R}$ and
        the fourth inequality is because $R>\frac{b}{2\lambda}$. This
        completes the proof of~\eqref{eq:Lyapunov-condition-2}. 
	
\medskip

	Proof of~\eqref{eq:minorization-2}: Letting $A=K$ in the
        minorization condition~\eqref{eq:assumption2},
	\begin{equation*}
	P\mathds{1}_K(x) \geq \alpha \nu(K) \mathds{1}_K(x).
	\end{equation*}
	As a result, applying $P$ to both sides of~\eqref{eq:assumption2},
	\begin{align*}
	P^2 \mathds{1}_A(x) \geq \alpha \nu(A) P \mathds{1}_K(x) \geq \alpha^2\nu(K) \nu(A) \mathds{1}_K(x),
	\end{align*}
	which proves~\eqref{eq:minorization-2}. 
	
Spectral gap~\eqref{eq:spec-gap-result}
follows from application of Corollary~\ref{cor:positive} to positive semi-definite operator $P^2$ satisfying conditions~\eqref{eq:Lyapunov-condition-2}-\eqref{eq:minorization-2}.
\end{proof}

\subsection{Spectral gap for a non-reversible chain}\label{sec:NR}

Suppose $P$ is a  Markov operator with a unique invariant measure $\pi$ and
suppose $P^\dagger$ is its adjoint in $L^2(\pi)$.  Then $P^\dagger
P$ is a Markov operator with a reversible invariant measure
$\pi$.
%Amir what about uniqueness?
  
\begin{proposition}\label{prop:non-reversible}
Suppose both $P$ and its adjoint $P^\dagger$ satisfy
condition (v4) (inequality~\eqref{eq:Lyapunov-condition}-\eqref{eq:assumption2}) with the same Foster Lyapunov function $V$, set $K$ and
constants $\lambda$, $b$ and $\alpha$.  Then $P^{\dagger}P$ satisfies 
\begin{align}
P^{\dagger} P V &\leq (1-\lambda')V + b' \mathds{1}_K, \label{eq:Lyapunov-condition-3}\\
P^{\dagger} P\mathds{1}_A(x) &\geq \alpha' \nu(A) \mathds{1}_{K}(x),\quad \forall A \in \mathcal{B}\label{eq:minorization-3}
\end{align}
where $\lambda'=\lambda(\frac{3}{2}-\lambda)$, $b'=(2-\lambda) b$, and
$\alpha' = \alpha^2\nu(K)$.  Consequently,
\begin{equation}\label{eq:spec-gap-result-3}
\|P\|_{L^2_0(\pi)}\leq \left(1-\beta^+\right)^{\frac{1}{2}}.
\end{equation}
with
$\beta^+=\frac{\lambda(\frac{3}{2}-\lambda)}{1+\frac{2b(2-\lambda)}{\alpha^2\nu(K)}}$.  
\end{proposition}

\medskip

\begin{proof}
%Amir -- write a complete proof.  You may skip a few calculation
%steps but describe all steps.    
The proof of~\eqref{eq:Lyapunov-condition-3}-\eqref{eq:minorization-3}
is entirely analogous to the proof of Proposition~\ref{lem:P-squared}. It follows
from applying $P^\dagger$ to the
inequalities~\eqref{eq:Lyapunov-condition}-\eqref{eq:assumption2} upon
using the fact that $P^\dagger$ also satisfies the
inequalities~\eqref{eq:Lyapunov-condition}-\eqref{eq:assumption2}. 
The bound~\eqref{eq:spec-gap-result-3} then follows from
%AMIR: check.  I changed Theorem 1 to Corollary 1.
application of Corollary~\ref{thm:Lyapunov-Poincare} to $P^\dagger P$ which is 
reversible, positive-definite, and satisfies the  conditions~\eqref{eq:Lyapunov-condition-3}-\eqref{eq:minorization-3}.
% Therefore, % according to Theorem~\ref{thm:Lyapunov-Poincare}, Poincar\'e inequality~\eqref{eq:poincare_+} holds with $P$ replaced by $P^\dagger P$, 
%\begin{equation*}
%\lr{f}{(I-P^\dagger P )f}\geq \beta_+ \|f\|_{2,\pi},\quad \forall f\in L^2_0(\pi),
%\end{equation*}
%%This implies the spectral gap $\|P\|_{L^2_0(\pi)} \leq (1-\beta_+)^{\frac{1}{2}}$ because
%which implies 
%%the spectral gap
%\begin{align*}
%\|Pf\|^2_{2,\pi}&=\lr{f}{P^\dagger P f}_\pi  
%\leq (1-\beta_+) \|f\|^2_{2,\pi}.
%\end{align*}
%concluding spectral gap~\eqref{eq:spec-gap-result-3}. 
\end{proof}

\section{Examples}\label{sec:example}
\subsection{Ornstein-Uhlenbeck process}
Consider the discrete-time Markov chain $\{X_n\}_{n\geq 0}$
taking values in $\Re$ that evolves according to
\begin{equation*}
X_{n+1} = (1-a)X_n + \sigma B_n,
\end{equation*}
where $a\in (0,1)$, $\sigma>0$ and $\{B_n\}_{n\geq 0}$  are
independent Gaussian random variables.  The associated Markov operator 
\begin{align*}
Pf(x) &= \Expect[f((1-a)x+\sigma B_1)]
%\\&=\int_{\Re}\frac{e^{-\frac{(y-(1-a)x)^2}{2\sigma^2}}}{\sqrt{2\pi\sigma^2}}f(y)\ud y
\\&=\int_{\Re}(2\pi\sigma^2)^{-\frac{1}{2}} \exp(-\frac{(y-(1-a)x)^2}{2\sigma^2})f(y)\ud y,
%\\
%&= \int_{\Re}\frac{e^{-(1-a)^2\frac{(y-x)^2}{2\sigma^2}}}{\sqrt{2\pi\sigma^2}}f(y) \pi(y)\ud y,
\end{align*}
with a reversible Gaussian invariant measure 
\begin{align*}
\ud \pi(x)=(\frac{2\pi \sigma^2}{a^2})^{-\frac{1}{2}}\exp({-\frac{x^2}{2\frac{\sigma^2}{a^2}}})\ud x.
\end{align*}
The Markov operator $P$ is an example of the Ornstein-Uhlenbeck Diffusion semigroup~\cite[Sec. 2.7.1]{bakry2013} with spectrum 
\begin{equation*}
\lambda_n = (1-a)^n,\quad n=0,1,\ldots
\end{equation*}
yielding the spectral gap $\|P\|_{L^2_0(\pi)}=1-a$. 

Our goal in this example is to apply the results of this paper to
obtain a bound for the spectral gap of $P$ and compare it to the exact
spectral gap.  Consider the Lyapunov function $V(x)= 1+x^2$. Then,
\begin{align*}
PV(x) &= 1+ (1-a)^2x^2 + \sigma^2 \\&\leq (1-a) V + (\sigma^2+1)\mathds{1}_{|x|\leq R}, 
\end{align*} 
with $R^2=\frac{\sigma^2+1}{a(1-a)}$. The minorization condition~\eqref{eq:assumption2} also holds:
\begin{align*}
P\mathds{1}_A(x) &= \mathbb P\{(1-a)x + \sigma B_1 \in A \} \\&\geq \alpha \mathbb P\{\frac{\sigma}{\sqrt{2}}B_1 \in A\} \mathds{1}_{|x|\leq R},
\end{align*}
where $\alpha=\exp(-\frac{\sigma^2+1}{\sigma^2}\frac{1-a}{a})$. 
%Hence, the
%conditions~\eqref{eq:Lyapunov-condition}-\eqref{eq:assumption2} of
%Assumption~\ref{asmpt:Lyapunov} hold. 
Since $P$ is a positive definite operator on $L^2(\pi)$,
Corollary~\ref{cor:positive} applies and one obtains 
% the spectral gap
\begin{equation*}
\|P\|_{L^2_0(\pi)} \leq  1-\frac{a}{1+(\sigma^2+1)\exp(\frac{\sigma^2+1}{\sigma^2}\frac{1-a}{a})}.
\end{equation*}
This is a conservative bound based on the exact spectral gap. The
bound may be improved with another choice of Lyapunov function
(e.g., $\exp(|x|)$). In general, it is known that the Lyapunov method
is only able to provide a conservative bound for the spectral gap; cf.~\cite[pp. 203]{bakry2013}.

\subsection{Diffusion map}
The diffusion map $T_\epsilon$ is a Markov operator defined as
\begin{equation}\label{eq:Teps}
T_\epsilon (f) (x):=  \frac{\int_{\Re^d} \geps(x-y)e^{-U(y)}f(y)\ud y }{ \int_{\Re^d} \geps(x-y)e^{-U(y)}\ud y},
\end{equation}
where $\geps(z) = \exp^{-\frac{|z|^2}{4\epsilon}}$ is the Gaussian
kernel, $U(x)$ is a potential function of sufficient regularity, and
$\epsilon>0$ is a positive parameter. The diffusion map was introduced
and studied in spectral clustering  literature as asymptotic limit of
graph Laplacian matrix~\cite{coifman,hein-consistency-2005}. Explicit
bounds on the spectral gap of $\Teps$ are important for analysis of
diffusion map-based algorithms such as the gain function approximation
algorithm in
the feedback particle filter.~\cite{taghvaei2019diffusion}.

%Importantly, $\Teps$ can be used to approximate the Markov diffusion semigroup associated with the stochastic differential equation
%\begin{equation*}
%\ud X_t  = -2\nabla U(X_t) \ud t +\sqrt{2} \ud B_t,
%\end{equation*}
%where $B_t$ is the standard Brownian motion~\cite{coifman,taghvaei2019diffusion}. 

The spectral gap is obtained via an application of
Corollary~\ref{cor:positive}. 
It is straightforward to check that $\Teps$ is a Markov operator with
reversible invariant probability density 
\[\pi (x) := \gamma
  e^{-U(x)}\int \geps(x-y)e^{-U(y)} \ud y,
\]  
where $\gamma$ is the normalization constant.  Moreover, $\Teps$ is positive-definite because 
\begin{equation*}
\lr{f}{\Teps f}_{\pi} = \int_{\Re^d} \geps(x-y) (e^{-U}f)(x)(e^{-U}f)(y) \ud x  \ud y \geq 0,
\end{equation*}
for all $f\in L^2(\pi)$.
% , where the positive-definiteness of the heat
% semigroup $f \mapsto \int \geps(x-y)f(y)\ud y$ is used. 
It remains to verify the Lyapunov
conditions~\eqref{eq:Lyapunov-condition}-\eqref{eq:assumption2}. This
requires additional assumptions on the potential function $U(x)$, an
example of which appears in the following Proposition. 
%The proof appears in the Appendix.
\begin{proposition}\label{prop:Teps}
	Consider the diffusion map operator~\eqref{eq:Teps}.  Suppose
        $U$ is bounded from below and twice continuously differentiable
        with a bounded Hessian $\|\nabla^2 U\|_\infty < \infty $. % where $\nabla^2 U$ is the Hessian of $U$
Also, suppose $\exists \; \lambda_0,R>0$ such that
	\begin{equation}\label{eq:assump-U}
	\frac{1}{2}|\nabla U(x)|^2 \geq  \lambda_0 U(x) + \|\nabla ^2U\|_\infty,~~ \forall~ |x|\geq R 
	\end{equation}
	Then, for all $\epsilon \in(0,\frac{1}{4\|\nabla^2 U\|_\infty })$ the Lyapunov conditions~\eqref{eq:Lyapunov-condition}-\eqref{eq:assumption2} hold and $\Teps$ admits a spectral gap  
	\begin{equation}\label{eq:gap-Teps}
	\|\Teps\|_{L^2_0(\pi)}\leq 1- \frac{\epsilon \lambda_0 }{1+\frac{2 \epsilon b_0 }{\alpha}}
	\end{equation}
	where 
%	satisfies with constants 
	\begin{align*}
%	\lambda&= \epsilon \kappa, \\
	b_0&= \|\nabla^2 U\|_\infty + \max_{\|x\|\leq R} \{\lambda_0 U(x) - \frac{1}{2}|\nabla  U(x)|^2\},\\
	\alpha &= \min_{|x|\leq R} \frac{1}{\sqrt{2}}e^{-2\epsilon|\nabla U(x)|^2 -3\epsilon\|\nabla^2 U\|_\infty - \frac{2}{\sigma^2}(x-\sigma^2\nabla U(x))^2 }. 
	\end{align*}
	and 
%	the set $K=\{x\in \Re^d;~\|x\|\leq R\}$, probability measure $\nu(A) = \mathbb{P}\{\frac{\sigma}{\sqrt{2}}B_1 \in A\}$, and 
$\sigma^2=\frac{2\epsilon}{1+2\|\nabla ^2 U\|_\infty \epsilon}$.  
\end{proposition}
\medskip

\begin{remark}
The assumption~\eqref{eq:assump-U} on $U$ is a type of a dissipative
condition for a dynamical systems with drift $\nabla U$~\cite{hale1988asymptotic}. 
It is satisfied by any potential function $U$ that has a quadratic growth
as $|x|\to\infty$. For example, $U(x) =  U_0(x) + \frac{1}{2}\delta
|x|^2$ satisfies this assumption provided $U_0$ is Lipschitz and
$\delta>0$.
\end{remark}
% It is possible to obtain better constants if one considers $T^n_{\frac{1}{n}}$ for  large $n\in \mathbb{N}$.  
%The spectral gap for $\Teps$ is critical in the analysis of diffusion map-based algorithms~\cite{taghvaei2019diffusion}. 

\medskip

\begin{proof}
Without loss of generality assume $U(x)\geq  1$ for all $x$. The
Lyapunov  condition~\eqref{eq:Lyapunov-condition} holds with $V=U$ because
\begin{align*}
 (\Teps & U) (x) \overset{(1)}{\leq } \log( (\Teps e^{U}) (x) ) = -\log (\int \geps(x-y)e^{-U(y)} \ud y)\\
&\overset{(3)}{\leq }  -\log (\int \geps(x-y)e^{-U(x) - \lr{\nabla U(x)}{y-x} - \frac{m}{2}|y-x|^2 } \ud y)\\
&=U(x) - \frac{\sigma^2}{2}|\nabla U(x)|^2 - \frac{1}{2}\log(\frac{\sigma^2}{2\epsilon})
\end{align*} 
where the Jensen's inequality  is used in the first step, and the inequality $U(y)\leq U(x) + \lr{\nabla U(x)}{y-x} + \frac{m}{2}|y-x|^2$ is used in the third step, with 
 $m = \|\nabla^2U\|_\infty$  and $\sigma^2 = \frac{2\epsilon}{1+2m\epsilon}$. Then, using~\eqref{eq:assump-U} 
 \begin{align*}
 (\Teps U) (x) &\leq U(x) - \sigma^2 \lambda_0 U(x) - \frac{1}{2}\log(\frac{\sigma^2}{2\epsilon} ) - \sigma^2 m  \\&\leq (1-\epsilon \lambda_0)U(x),\quad \text{if}~~|x|\geq R
 \end{align*} 
 where $\sigma^2 \geq \epsilon$ and $\log(\frac{\sigma^2}{2\epsilon}) \geq -2\epsilon m$ are used in the second step. This proves Lyapunov condition~\eqref{eq:Lyapunov-condition}, with $\lambda = \epsilon \lambda_0$, $K=\{x\in \Re^d;~|x|\leq R\}$, and $b=\epsilon b_0$. 
 
 The minorization condition~\eqref{eq:assumption2} holds because 
 \begin{align*}
 \Teps \mathds{1}_A(x) &= \frac{\int_A \geps(x-y)e^{-U(y)}\ud y}{\int_{\Re^d} \geps(x-y)e^{-U(y)}\ud y}\\
 &\geq \frac{e^{\frac{\epsilon}{2} |\nabla U(x)|^2-
   \frac{m\epsilon}{2}} {\sf P}(x-\sigma^2\nabla U(x) + \sigma B_1 \in A) }{e^{2\epsilon|\nabla U(x)|^2 + 2m\epsilon}}\\
 &\geq \alpha {\sf P}(\frac{\sigma}{\sqrt{2}} B_1\in A ),\quad \text{if}~~|x|\leq R
 \end{align*}
 
 This proves the Lyapunov conditions~\eqref{eq:Lyapunov-condition}-\eqref{eq:assumption2}, which together with the fact that $\Teps$ is reversible and positive-definite, proves the spectral gap~\eqref{eq:gap-Teps} by Corollary~\ref{cor:positive}.
\end{proof}

%{
%\color{blue} 
%\subsection{OMIT? Discrete-time Langevin equation}
%This example serves to illustrate the limitations of the application of the results presented in this paper to non-reversible Markov chains.  
%Consider the discrete-time Markov chain $\{X_n\}_{n\geq 0}$ governed by the discretized Langevin equation
%\begin{equation}
%X_{n+1} = X_n - \Delta t \nabla U(X_n) + \sqrt{2\Delta t}B_n
%\end{equation}
%where $\Delta t$ is the time step-size, $U$ is the potential function, and $\{B_n\}_{n\geq 0}$ are independent standard Gaussian random variables. Under suitable conditions on $U$, the Markov chain admits a unique stationary distribution~\cite[Prop. 2]{durmus2019high}. However, the Markov chain is not reversible and the results of of Section~\ref{sec:main} are not directly applicable. It is also difficult to prove the Lyapunov condition for the adjoint operator and apply Proposition~\ref{prop:non-reversible}. A way to get around this is to consider the adjusted Langevin algorithm~\cite{rossky1978brownian} which introduces an additional Metropolis-Hasting sampling step in the algorithm. With the adjustment, the Markov-chain is reversible,  and it is possible to show an spectral gap in $L^2$ by establishing the Lyapunov condition for the adjusted process. Note that the existing results such as~\cite{hairer2011yet} do not have this limitation, because they are applicable to non-reversible Markov chains as well.  
%}
\section{Conclusion}\label{sec:conclusion}
In this paper, a straightforward analytical approach is presented to
establish stochastic stability starting from the Lyapunov drift
condition (v4) (Theorem~\ref{thm:Lyapunov-Poincare}).  A key message
of this paper is that the Lyapunov Foster drift condition (v4), in of
itself, only implies a bound on the spectral gap from eigenvalue at
$1$.  This is formalized in this paper as a relationship between
condition (v4) and the Poincare inequality for the operator $I-P$
(Theorem~\ref{thm:Lyapunov-Poincare}).  From this main theorem, two sets of bounds are
obtained here under certain hypotheses (Corollary~\ref{cor:positive} and
Proposition~\ref{lem:P-squared}).  An extension to the non-reversible
chain is also described in Proposition~\ref{prop:non-reversible}.  Two
illustrative examples are presented.  The diffusion map example is of
independent interest for analysis of gain function approximation in
the feedback particle filter.

\bibliographystyle{plain}
\bibliography{MC-stability-ref} 
%\bibliography{bibfiles/ref,bibfiles/fpfbib,bibfiles/fpfbib2,bibfiles/fpfrelated,bibfiles/meanfield,bibfiles/meanfield_v2,bibfiles/Optimization}

\end{document}